\documentclass{amsart}
\usepackage[T1]{fontenc}
\usepackage[utf8]{inputenc}
\usepackage[english]{babel}
\usepackage{amsmath}
\usepackage{amsfonts}
\usepackage{amssymb}
\usepackage{amsthm}
\usepackage[foot]{amsaddr}
\usepackage{mathtools}
\mathtoolsset{showonlyrefs=true}
\usepackage[colorlinks=true,allcolors=blue]{hyperref}%use it for following links during development
\usepackage{verbatim}
\usepackage{physics}
\usepackage{esint}
\usepackage{enumitem}
\usepackage{float}
\usepackage{subcaption}
\usepackage{multirow}
\usepackage{array}
\usepackage{booktabs}
\usepackage{threeparttable}
\usepackage{tikz}
\DeclareMathOperator*{\essinf}{ess\ inf}
\DeclareMathOperator*{\esssup}{ess\ sup}

\allowdisplaybreaks

\calclayout

\newtheorem{thm}{Theorem}[section]

\newtheorem{lemma}[thm]{Lemma}
\newtheorem{prop}[thm]{Proposition}
\newtheorem{cor}[thm]{Corollary}
\newtheorem{rem}[thm]{Remark}

\let\epsilon\varepsilon

\begin{document}
\author{Mihály Kovács$^{1,2,3}$}
\thanks{M. Kov\'acs acknowledges the support of the Hungarian National Research, Development and Innovation Office (NKFIH) through Grant no. TKP2021-NVA-02 and Grant no. K-145934.}
\author{Mihály A. Vághy$^3$}
\thanks{M. A. Vághy acknowledges the support of the ÚNKP-23-3-II-PPKE-81 National Excellence Program of the Ministry for Culture and Innovation from the source of the National Research, Development and Innovation Fund.}
\thanks{This article is based upon work from COST Action CA18232 MAT-DYN-NET, supported by COST (European Cooperation in Science and Technology).}
\address{$^1$Department of Mathematical Sciences, Chalmers University of Technology and University of Gothenburg, SE-41296 Gothenburg, Sweden}
\address{$^2$Department of Analysis and Operations Research, Budapest University of Technology and Economics, Műegyetem rkp. 3-9, H-1111 Budapest, Hungary}
\address{$^3$Faculty of Information Technology and Bionics, Pázmány Péter Catholic University, Práter u. 50/a, H-1444 Budapest, Hungary}

\title{Neumann-Neumann type domain decomposition of elliptic problems on metric graphs}

\begin{abstract}
	In this paper we develop a Neumann-Neumann type domain decomposition method for elliptic problems on metric graphs. We describe the iteration in the continuous and discrete setting and rewrite the latter as a preconditioner for the Schur complement system. Then we formulate the discrete iteration as an abstract additive Schwarz iteration and prove that it convergences to the finite element solution with a rate that is independent of the finite element mesh size. We show that the condition number of the Schur complement is also independent of the finite element mesh size. We provide an implementation and test it on various examples of interest and compare it to other preconditioners.
\end{abstract}

\maketitle

\section{Introduction}
In recent decades differential operators on metric graphs have found a myriad of applications when describing quasi-one-dimensional phenomena in a broad range of fields, such as superconductivity in granular materials \cite{Alexander1983}, classical wave propagation in wave guide networks \cite{Flesia1987,Flesia1989}, membrane potential of neurons \cite{Kallianpur1984}, cell differentiation \cite{Cho2018}, and optimal control \cite{Avdonin2023,Avdonin2019,Mehandiratta2021,Stoll2021}. 

We consider a quantum graph; that is, a metric graph $\mathsf{G}$ equipped with an elliptic differential operator on each edge and certain standard vertex conditions. The graph consists of a finite set $\mathsf{V}$ of vertices and a finite set $\mathsf{E}$ of edges connecting pairs of vertices. We assume that the graph is simple and does not contain parallel edges or loops. Let $n=|\mathsf{V}|$ denote the number of vertices and $m=|\mathsf{E}|$ the number of edges. We assume that the graph is directed; that is, each edge has a specified (but otherwise arbitrary) orientation, and thus an origin and a terminal vertex. Each edge $\mathsf{e}\in\mathsf{E}$ is assigned a length $\ell_{\mathsf{e}}\in(0,\infty)$ and a local coordinate $x\in[0,\ell_{\mathsf{e}}]$.% The volume of a quantum graph is defined as $\mathrm{vol}(\mathsf{G})=\sum_{\mathsf{e}\in\mathsf{E}}\ell_{\mathsf{e}}$.

A function $u$ on a metric graph $\mathsf{G}$ can be defined as a vector of functions and we write $u=(u_{\mathsf{e}})_{\mathsf{e}\in\mathsf{E}}$, and consider it to be an element of a product function space, to be specified later. Let $u_{\mathsf{e}}(\mathsf{v})$ denote the value of $u$ at $\mathsf{v}\in\mathsf{V}$ along the edge $\mathsf{e}\in\mathsf{E}$.

To define the vertex conditions, let us denote by $\mathsf{E_v}$ the set of edges incident to the vertex $\mathsf{v}\in\mathsf{V}$, and by $d_{\mathsf{v}}=|\mathsf{E_v}|$ the degree of $\mathsf{v}\in\mathsf{V}$. We denote by $\mathrm{int}(\mathsf{G})$ the set of vertices with degree $d_{\mathsf{v}}>1$ and by $\partial\mathsf{G}$ the set $\mathsf{V}\backslash\mathrm{int}(\mathsf{G})$. We seek solutions that are continuous on $\mathsf{G}$
and satisfy the Neumann-Kirchhoff (often called standard) condition, given as
\begin{equation}
    \sum_{\mathsf{e}\in\mathsf{E}}u_{\mathsf{e}}'(\mathsf{v})=0,\qquad\mathsf{v}\in\mathsf{V},
\end{equation}
where the derivatives are assumed to be taken in the directions away from the vertex. If $d_{\mathsf{v}}=1$, then this reduces to the classical zero Neumann boundary condition.

In order to write the vertex conditions more compactly, let us define the vector of function values at $\mathsf{v}\in\mathsf{V}$ as
\begin{equation}
    U(\mathsf{v})=\qty\big(u_e(\mathsf{v}))_{\mathsf{e}\in\mathsf{E_v}}\in\mathbb{R}^{d_{\mathsf{v}}}
\end{equation}
and the bi-diagonal matrix
\begin{equation}
    I_{\mathsf{v}}=\begin{bmatrix}1&-1&&\\&\ddots&\ddots&\\&&1&-1\end{bmatrix}\in\mathbb{R}^{(d_{\mathsf{v}}-1)\times d_{\mathsf{v}}}.
\end{equation}
Then $I_{\mathsf{v}}U(\mathsf{v})=0\in\mathbb{R}^{d_{\mathsf{v}}-1}$ implies that the function values along the edges in $\mathsf{E_v}$ coincide at $\mathsf{v}\in\mathsf{V}$. Similarly, we define
\begin{equation}
    U'(\mathsf{v})=\qty\big(u_e'(\mathsf{v}))_{\mathsf{e}\in\mathsf{E_v}}\in\mathbb{R}^{d_{\mathsf{v}}},
\end{equation}
the vector of function derivative at $\mathsf{v}\in\mathsf{V}$ and the row vector
\begin{equation}
    C(\mathsf{v})^{\top}=\qty\big(c_e(\mathsf{v}))_{\mathsf{e}\in\mathsf{E_v}}^{\top}\in\mathbb{R}^{1\times d_{\mathsf{v}}}.
\end{equation}
Then $C(\mathsf{v})^{\top}U'(\mathsf{v})=0$ implies that the function $u$ satisfies the Neumann-Kirchhoff conditions at $\mathsf{v}\in\mathsf{V}$.

Then the quantum graphs can be formally written as
\begin{equation}\label{eq:qg_PDE}
    \left\{\begin{aligned}
        -(c_{\mathsf{e}} u_{\mathsf{e}}')'(x)+p_{\mathsf{e}}(x) u_{\mathsf{e}}(x)&=f_{\mathsf{e}}(x),&x\in(0,\ell_{\mathsf{e}}),~\mathsf{e}\in\mathsf{E},~~ & (a)\\
        0&=I_{\mathsf{v}}U(\mathsf{v}),&\mathsf{v}\in\mathrm{int}(\mathsf{G}),~~& (b)\\
        0&=C(\mathsf{v})^{\top}U'(\mathsf{v}),&\mathsf{v}\in\mathsf{V},~~& (c).
    \end{aligned}
    \right.
\end{equation}

We wish to approximate the solution of \eqref{eq:qg_PDE} in the finite element framework. In \cite{Arioli2017} a special finite element is assigned to the vertices that have a star shaped support on the neighbouring edges ensuring the continuity of solutions, and use standard finite elements on the edges. Then the authors prove usual error estimates and an upper bound of the Neumann-Kirchhoff residual of the discrete solution. However, the size of the corresponding stiffness matrix can quickly grow and it loses its banded (tridiagonal) nature. We also note that in applications where the graph itself is time-varying (for example if a road is blocked in a traffic network) it might be expensive to modify the structure of the stiffness matrix.

To overcome such issues, we investigate a Neumann-Neumann type nonoverlapping domain decomposition method. The mathematical background of overlapping domain decomposition methods originate from \cite{Schwarz1870}, which was further developed in \cite{Babuska1957,Morgenstern1956,Sobolev1936}. Later nonoverlapping methods gained attention due to their natural parallelism and efficiency in numerical applications along with the growth of high performance computing \cite{Dryja1987,Lions1988,Lions1989}. Many variants have been developed since, such as Lagrange multiplier based Finite Element Tearing and Interconnecting (FETI) methods \cite{Farhat1991,Klawonn2001}, least squares-control methods \cite{Gunzburger2000,Lions1981}, and multilevel or multigrid methods \cite{Bank1988,Brandt1977,Briggs2000}. In particular, Neumann-Neumann methods can be traced back to \cite{Bourgat1989,Roeck1991,Dihn1984,Tallec1991}. For introductory surveys we refer to \cite{Chan1994,Xu1992}, see also \cite[Chapter 7]{Brenner2007}, while more thorough theoretical background and historical overview can be found in \cite{Mathew2008,Smith1996,Toselli2005}. While certain domain decomposition methods have been successfully designed and applied for optimal control on networks \cite{Leugering1998,Leugering2000,Leugering2017,Leugering2023} and its theory was established in \cite{Lagnese2004}, to the authors knowledge, the performance and the convergence of Neumann-Neumann type iterative substructuring methods was never addressed. First, we rewrite the method as a preconditioner for the Schur complement system, then rigorously show via the abstract additive Schwarz framework that the iteration converges to the finite element solution with a geometric rate that is independent of the finite element mesh size. While preparing for this proof we show that the condition number of the underlying Schur complement is also indepenedent of the finite element mesh size.

The paper is organized as follows. Section 2 contains a brief overview of the abstract problem, the corresponding weak formulation and its FEM solution, and the abstract additive Schwarz framework. In Section 3 we introduce the Neumann-Neumann method and prove its convergence to the FEM solution through the Schwarz framework. We also formulate the method as a preconditioner to the Schur complement system. We note because of the quasi-one-dimensional nature of the problem we can use powerful tools like Sobolev's embedding, and thus our proofs are much simpler and more transparent then that of classical domain decomposition methods in two or more dimensions. Finally, in Section 4, we demonstrate the strength of our approach through various examples and compare it to other preconditioners.

\section{Preliminaries}
Let $L^2(a,b)$ be the Hilbert space of real-valued square-integrable functions equipped with the norm
\begin{equation}
    \norm{f}_{L^2(a,b)}^2=\int_a^b\qty\big|f(x)|^2\dd{x},\qquad f\in L^2(a,b),
\end{equation}
and $L^{\infty}(a,b)$ be the Banach space of real-valued essentially bounded functions equipped with the norm
\begin{equation}
    \norm{f}_{L^\infty(a,b)}=\esssup_{x\in(a,b)}\qty\big|f(x)|,\qquad f\in L^{\infty}(a,b),
\end{equation}
and $H^k(a,b)$ be the Sobolev space of real-valued square-integrable functions whose generalized derivatives up to the $k$th order are also square-integrable, equipped with the norm
\begin{equation}
    \norm{f}_{H^k(a,b)}^2=\sum_{j=0}^k\norm{f^{(j)}}_{L^2(a,b)}^2,\qquad f\in H^k(a,b),
\end{equation}
and $C[a,b]$ be the Banach space of real-valued continuous functions equipped with the supremum norm. Using these, we define the Banach spaces
\begin{equation}
    L^2(\mathsf{G})=\bigoplus_{\mathsf{e}\in\mathsf{E}}L^2(0,\ell_{\mathsf{e}}),\qquad L^{\infty}(\mathsf{G})=\bigoplus_{\mathsf{e}\in\mathsf{E}}L^{\infty}(0,\ell_{\mathsf{e}}),\qquad H^k(\mathsf{G})=\bigoplus_{\mathsf{e}\in\mathsf{E}}H^k(0,\ell_{\mathsf{e}}).
\end{equation}
endowed with the natural norms
\begin{equation}
    \begin{aligned}
        &\norm{u}_{L^2(\mathsf{G})}^2:=\sum_{\mathsf{e}\in\mathsf{E}}\norm{u_{\mathsf{e}}}_{L^2(0,\ell_{\mathsf{e}})}^2,\qquad&&u=(u_{\mathsf{e}})_{\mathsf{e}\in\mathsf{E}}\in L^2(\mathsf{G}),\\
        &\norm{u}_{L^{\infty}(\mathsf{G})}^2:=\max_{\mathsf{e}\in\mathsf{E}}\norm{u_{\mathsf{e}}}_{L^{\infty}(0,\ell_{\mathsf{e}})},\qquad&&u=(u_{\mathsf{e}})_{\mathsf{e}\in\mathsf{E}}\in L^{\infty}(\mathsf{G}),\\
        &\norm{u}_{H^k(\mathsf{G})}^2:=\sum_{\mathsf{e}\in\mathsf{E}}\norm{u_{\mathsf{e}}}_{H^k(0,\ell_{\mathsf{e}})}^2,\qquad&&u=(u_{\mathsf{e}})_{\mathsf{e}\in\mathsf{E}}\in H^k(\mathsf{G}).
    \end{aligned}
\end{equation}
We note that the spaces $L^2(\mathsf{G})$ and $H^k(\mathsf{G})$ are Hilbert spaces with the natural inner products. Finally, we define the space of continuous functions defined on $\mathsf{G}$ as
\begin{equation}
    C(\mathsf{G}):=\qty\Big{u=(u_\mathsf{e})_{\mathsf{e}\in\mathsf{E}}\Big|I_{\mathsf{v}}U(\mathsf{v})=0,~\forall \mathsf{e}\in\mathsf{E}:u_{\mathsf{e}}\in C[0,\ell_{\mathsf{e}}]}.
\end{equation}

\subsection{The abstract problem}
On $L^2(\mathsf{G})$ we define the elliptic operator
\begin{equation}
    \mathcal{A}_{\mathrm{max}}:=\mathrm{diag}\qty\Bigg(-\dv{}{x}\qty\bigg(c_{\mathsf{e}}\dv{}{x})+p_{\mathsf{e}})_{\mathsf{e}\in\mathsf{E}},\qquad D(\mathcal{A}_{\mathrm{max}})=H^2(\mathsf{G}).
\end{equation}
We further define the feedback operator $\mathcal{B}:D(\mathcal{A}_{\mathrm{max}})\mapsto\mathcal{Y}$ by
\begin{equation}
    \mathcal{B}u=\begin{bmatrix}\qty\big(I_{\mathsf{v}}U(\mathsf{v}))_{\mathsf{v}\in\mathsf{V}}\\\qty\big(C(\mathsf{v})^{\top}U'(\mathsf{v}))_{\mathsf{v}\in\mathsf{V}}\end{bmatrix},\qquad D(\mathcal{B})=D(\mathcal{A}_{\mathrm{max}}),
\end{equation}
where $\mathcal{Y}=\ell^2(\mathbb{R}^{2m})\eqsim\mathbb{R}^{2m}$. Finally, we define
\begin{equation}
    \mathcal{A}:=\mathcal{A}_{\mathrm{max}},\qquad D(\mathcal{A}):=\qty\big{u\in D(\mathcal{A}_{\mathrm{max}}):~\mathcal{B}u=0_{\mathcal{Y}}}.
\end{equation}
Throughout the paper we assume that $c=\qty\big(c_{\mathsf{e}})_{\mathsf{e}\in\mathsf{E}}:\mathsf{G}\mapsto\mathbb{R}$ is a positive Lipschitz function, that the function $p=\qty\big(p_{\mathsf{e}})_{\mathsf{e}\in\mathsf{E}}\in L^{\infty}(\mathsf{G})$ satisfies $\essinf_{x\in\mathsf{G}}~p(x)\ge p_0$ for some $p_0>0$, and that $f=\qty\big(f_{\mathsf{e}})_{\mathsf{e}\in\mathsf{E}}\in L^2(\mathsf{G})$. Using this, we can reformulate \eqref{eq:qg_PDE} as follows: find $u\in D(\mathcal{A})$ such that
\begin{equation}\label{eq:|PDE|}
    \mathcal{A}u=f.
\end{equation}

\subsection{Weak formulation and FEM}
While \eqref{eq:|PDE|} is well-posed \cite[Proposition 3.1]{Bolin2023}, for our purposes it is convenient to introduce a weak formulation of \eqref{eq:qg_PDE}. The corresponding bilinear form $\mathfrak{a}:H^1(\mathsf{G})\times H^1(\mathsf{G})\mapsto\mathbb{R}$ is defined as
\begin{equation}
    \begin{aligned}
        \mathfrak{a}(u,v)&=\sum_{\mathsf{e}\in\mathsf{E}}\qty\Bigg(\int_{\mathsf{e}}c_{\mathsf{e}}(x)u_{\mathsf{e}}'(x)v_{\mathsf{e}}'(x)\dd{x}+\int_{\mathsf{e}}p_{\mathsf{e}}(x)u_{\mathsf{e}}(x)v_{\mathsf{e}}(x)\dd{x}),\\
        D(\mathfrak{a})&=\qty\big{u\in H^1(\mathsf{G}):I_{\mathsf{v}}U(\mathsf{v})=0,~\mathsf{v}\in\mathsf{V}},
    \end{aligned}
\end{equation}
see \cite[Lemma 3.3]{Mugnolo2006} and \cite[Lemma 3.4]{Mugnolo2007}. We highlight the the Neumann-Kirchhoff condition do not appear in this bilinear form or in its domain. Thus, we seek a solution $u\in D(\mathfrak{a})$ such that
\begin{equation}\label{eq:weak_PDE}
    \mathfrak{a}(u,v)=f(v),\qquad v\in D(\mathfrak{a}),
\end{equation}
where $f(v):=\langle f,v\rangle_{L^2(\mathsf{G})}$. It is well-known that under our assumptions the symmetric bilinear form $\mathfrak{a}(\cdot,\cdot)$ is bounded and coercive, and thus \eqref{eq:weak_PDE} is well-posed in light of the Riesz representation theorem. Moreover, the unique solution of \eqref{eq:weak_PDE} is the unique solution of \eqref{eq:|PDE|}.

Following \cite{Arioli2017} for the sake of notational simplicity we consider an equidistant discretization on the edges. This approach and our subsequent analysis can be trivially generalized to the nonequidistant case. We divide each edge $\mathsf{e}=(\mathsf{v}_a^{\mathsf{e}},\mathsf{v}_b^{\mathsf{e}})$ into $n_{\mathsf{e}}\ge2$ intervals of length $h_{\mathsf{e}}\in(0,1)$. For the resulting $\qty{x_j^{\mathsf{e}}}_{j=1,2,\dots,n_{\mathsf{e}}-1}$ nodes we introduce the standard basis $\qty{\psi_j^\mathsf{e}}_{j=1,2,\dots,n_{\mathsf{e}}-1}$ of hat functions
\begin{equation}\label{eq:edge_hats}
    \psi_j^{\mathsf{e}}(x)=\begin{cases}
        1-\frac{|x_j^{\mathsf{e}}-x|}{h_{\mathsf{e}}},\quad&\text{if }x\in\qty\big[x_{j-1}^{\mathsf{e}},x_{j+1}^{\mathsf{e}}],\\
        0,&\text{otherwise,}
    \end{cases}
\end{equation}
where $x_0^{\mathsf{e}}=\mathsf{v}_a^{\mathsf{e}}$ and $x_{n_{\mathsf{e}}}^{\mathsf{e}}=\mathsf{v}_b^{\mathsf{e}}$. These functions are a basis of the finite-dimensional space $V_h^{\mathsf{e}}\subset H_0^1(0,\ell_{\mathsf{e}})\cap C[0,\ell_{\mathsf{e}}]$ of piecewise linear functions.

To each $\mathsf{v}$ we assign a special hat function $\phi_{\mathsf{v}}$ supported on the neighbouring set $W_{\mathsf{v}}$ of the vertex defined as
\begin{equation}
    W_{\mathsf{v}}=\qty\Bigg(\bigcup_{\mathsf{e}\in\mathsf{E}:\mathsf{v}_a^{\mathsf{e}}=\mathsf{v}}\qty\big[\mathsf{v},x_1^{\mathsf{e}}])\cup\qty\Bigg(\bigcup_{\mathsf{e}\in\mathsf{E}:\mathsf{v}_b^{\mathsf{e}}=\mathsf{v}}\qty\big[x_{n_{\mathsf{e}}-1}^{\mathsf{e}},\mathsf{v}]).
\end{equation}
Then $\phi_{\mathsf{v}}$ is defined as
\begin{equation}
    \phi_{\mathsf{v}}(x^{\mathsf{e}})=\begin{cases}
        1-\frac{|x_{\mathsf{v}}^{\mathsf{e}}-x^{\mathsf{e}}|}{h_{\mathsf{e}}},\quad&\text{if }x^{\mathsf{e}}\in W_{\mathsf{v}},\\
        0,&\text{otherwise,}
    \end{cases}
\end{equation}
where $x_{\mathsf{v}}^{\mathsf{e}}$ is either $0$ or $\ell_{\mathsf{e}}$ depending on the orientation of the edge.

We define the space
\begin{equation}
    V_h(\mathsf{G})=\qty\Bigg(\bigoplus_{\mathsf{e}\in\mathsf{E}}V_h^{\mathsf{e}})\oplus\mathrm{span}\qty{\phi_{\mathsf{v}}}_{\mathsf{v}\in\mathsf{V}}
\end{equation}
of piecewise linear functions. Note, that $V_h(\mathsf{G})\subset H^1(\mathsf{G})\cap C(\mathsf{G})=D(\mathfrak{a})$ by construction. Any function $w_h\in V_h(\mathsf{G})$ is a linear combination of the basis functions:
\begin{equation}
    w_h(x)=\sum_{\mathsf{e}\in\mathsf{E}}\sum_{j=1}^{n_{\mathsf{e}}-1}\alpha_j^{\mathsf{e}}\phi_j^{\mathsf{e}}(x)+\sum_{\mathsf{v}\in\mathsf{V}}\beta_{\mathsf{v}}\phi_{\mathsf{v}}(x).
\end{equation}
Thus the solution of \eqref{eq:weak_PDE} can be approximated by finding $u_h\in V_h(\mathsf{G})$ such that
\begin{equation}\label{eq:discrete_PDE}
    \mathfrak{a}(u_h,v_h)=f(v_h),\qquad v_h\in V_h(\mathsf{G}).
\end{equation}
Equivalently, we can test only on the basis functions. Since the neighbouring set of distinct vertices are disjoint we have that
\begin{equation}\label{eq:discrete_PDE_formula}
    \begin{aligned}
        \mathfrak{a}(w_h,\psi_k^{\mathsf{e}})&=\sum_{\mathsf{e}\in\mathsf{E}}\sum_{j=1}^{n_{\mathsf{e}}-1}\alpha_j^{\mathsf{e}}\int_{\mathsf{e}}\qty\big(c_{\mathsf{e}}{\psi_j^{\mathsf{e}}}'{\psi_k^{\mathsf{e}}}'+p_{\mathsf{e}}\psi_j^{\mathsf{e}}\psi_k^{\mathsf{e}})\dd{x}\\
        &+\sum_{\mathsf{v}\in\mathsf{V}}\beta_{\mathsf{v}}\int_{\mathsf{e}}\qty\big(c_{\mathsf{e}}{\phi_{\mathsf{v}}}'{\psi_k^{\mathsf{e}}}'+p_{\mathsf{e}}\phi_{\mathsf{v}}\psi_k^{\mathsf{e}})\dd{x}=f(\psi_k^{\mathsf{e}}),\quad k=1,2,\dots,n_{\mathsf{e}-1},~\mathsf{e}\in\mathsf{E},\\
        \mathfrak{a}(w_h,\phi_{\mathsf{v}})&=\sum_{\mathsf{e}\in\mathsf{E}}\sum_{j=1}^{n_{\mathsf{e}}-1}\alpha_j^{\mathsf{e}}\int_{\mathsf{e}}\qty\big(c_{\mathsf{e}}{\psi_j^{\mathsf{e}}}'{\phi_{\mathsf{v}}}'+p_{\mathsf{e}}\psi_j^{\mathsf{e}}\phi_{\mathsf{v}})\dd{x}\\
        &+\sum_{\mathsf{v}\in\mathsf{V}}\beta_{\mathsf{v}}\int_{\mathsf{e}}\qty\big(c_{\mathsf{e}}{\phi_{\mathsf{v}}}'{\phi_{\mathsf{v}}}'+p_{\mathsf{e}}\phi_{\mathsf{v}}\phi_{\mathsf{v}})\dd{x}=f(\psi_k^{\mathsf{e}}),\quad\mathsf{v}\in\mathsf{V}.\\
    \end{aligned}
\end{equation}
Let us denote by
\begin{equation}
    \boldsymbol{u}=\begin{bmatrix}u_{\mathsf{E}}\\u_{\mathsf{V}}\end{bmatrix},\qquad u_{\mathsf{E}}=\begin{bmatrix}u^{\mathsf{e}_1}\\u^{\mathsf{e}_2}\\\vdots\\u^{\mathsf{e}_m}\end{bmatrix},\qquad u^{\mathsf{e}}=\begin{bmatrix}u_1^{\mathsf{e}}\\u_2^{\mathsf{e}}\\\vdots\\u_{n_e-1}^{\mathsf{e}}\end{bmatrix},\qquad u_{\mathsf{V}}=\begin{bmatrix}u_{\mathsf{v}_1}\\u_{\mathsf{v}_2}\\\vdots\\u_{\mathsf{v}_n}\end{bmatrix}
\end{equation}
the vector of values that define the finite element function
\begin{equation}
    u_h(x)=\sum_{\mathsf{e}\in\mathsf{E}}\sum_{j=1}^{n_{\mathsf{e}}-1}u_j^{\mathsf{e}}\phi_j^{\mathsf{e}}(x)+\sum_{\mathsf{v}\in\mathsf{V}}u_{\mathsf{v}}\phi_{\mathsf{v}}(x),
\end{equation}
and by
\begin{equation}
    \boldsymbol{f}=\begin{bmatrix}f_{\mathsf{E}}\\f_{\mathsf{V}}\end{bmatrix},\qquad f_{\mathsf{E}}=\begin{bmatrix}f^{\mathsf{e}_1}\\f^{\mathsf{e}_2}\\\vdots\\f^{\mathsf{e}_m}\end{bmatrix},\qquad f^{\mathsf{e}}=\begin{bmatrix}f_1^{\mathsf{e}}\\f_2^{\mathsf{e}}\\\vdots\\f_{n_e-1}^{\mathsf{e}}\end{bmatrix},\qquad f_{\mathsf{V}}=\begin{bmatrix}f_{\mathsf{v}_1}\\f_{\mathsf{v}_2}\\\vdots\\f_{\mathsf{v}_n}\end{bmatrix}
\end{equation}
the vector of values
\begin{equation}
    f_k^{\mathsf{e}}=\int_{\mathsf{e}}f\psi_k^{\mathsf{e}}\dd{x},\qquad f_{\mathsf{v}}=\int_{W_{\mathsf{v}}}f\phi_{\mathsf{v}}\dd{x}.
\end{equation}
Then \eqref{eq:discrete_PDE_formula} can be rewritten as
\begin{equation}\label{eq:discrete_PDE_matrix}
    A\boldsymbol{u}=\boldsymbol{f},
\end{equation}
where the stiffness matrix $A$ has block structure as follows:
\begin{equation}
    A=\begin{bmatrix}
        A_{\mathsf{E}} & A_{\mathsf{EV}}\\
        A_{\mathsf{VE}} & A_{\mathsf{V}}
    \end{bmatrix}+\begin{bmatrix}
        B_{\mathsf{E}} & B_{\mathsf{EV}}\\
        B_{\mathsf{VE}} & B_{\mathsf{V}}
    \end{bmatrix}.
\end{equation}
Here
\begin{enumerate}
    \item the matrix $A_{\mathsf{E}}=\mathrm{diag}(A_{\mathsf{e}})_{\mathsf{e}\in\mathsf{E}}$ is block diagonal and the entries of the tridiagonal matrix $A_{\mathsf{e}}$ are given by
        \begin{equation}
            [A_{\mathsf{e}}]_{jk}=\int_{\mathsf{e}}c_{\mathsf{e}}{\psi_j^{\mathsf{e}}}'{\psi_k^{\mathsf{e}}}'\dd{x},\qquad j,k=1,2,\dots,n_{\mathsf{e}}-1
        \end{equation}
    \item the entries of the blocks of $A_{\mathsf{EV}}^{\top}=A_{\mathsf{VE}}=(A_{\mathsf{e}})_{\mathsf{e}\in\mathsf{E}}$ are given by
        \begin{equation}
            [A_{\mathsf{e}}]_{\mathsf{v}k}=\int_{W_{\mathsf{v}}}c_{\mathsf{e}}{\phi_{\mathsf{v}}}'{\psi_k^{\mathsf{e}}}'\dd{x},\qquad k=1,2,\dots,n_{\mathsf{e}}-1,~\mathsf{v}\in\mathsf{V},
        \end{equation}
    \item the entries of the diagonal matrix $A_{\mathsf{V}}=\mathrm{diag}(A_{\mathsf{v}})_{\mathsf{v}\in\mathsf{V}}$ are given by
        \begin{equation}
            A_{\mathsf{v}}=\int_{W_{\mathsf{v}}}c_{\mathsf{e}}{\phi_{\mathsf{v}}}'{\phi_{\mathsf{v}}}'\dd{x},
        \end{equation}
    \item the matrix $B_{\mathsf{E}}=\mathrm{diag}(B_{\mathsf{e}})_{\mathsf{e}\in\mathsf{E}}$ is block diagonal and the entries of the tridiagonal matrix $B_{\mathsf{e}}$ are given by
        \begin{equation}
            [B_{\mathsf{e}}]_{jk}=\int_{\mathsf{e}}p_{\mathsf{e}}\psi_j^{\mathsf{e}}\psi_k^{\mathsf{e}}\dd{x},\qquad j,k=1,2,\dots,n_{\mathsf{e}}-1
        \end{equation}
    \item the entries of the blocks of $B_{\mathsf{EV}}^{\top}=B_{\mathsf{VE}}=(B_{\mathsf{e}})_{\mathsf{e}\in\mathsf{E}}$ are given by
        \begin{equation}
            [B_{\mathsf{e}}]_{\mathsf{v}k}=\int_{W_{\mathsf{v}}}p_{\mathsf{e}}\phi_{\mathsf{v}}\psi_k^{\mathsf{e}}\dd{x},\qquad k=1,2,\dots,n_{\mathsf{e}}-1,~\mathsf{v}\in\mathsf{V},
        \end{equation}
    \item the entries of the diagonal matrix $B_{\mathsf{V}}=\mathrm{diag}(B_{\mathsf{v}})_{\mathsf{v}\in\mathsf{V}}$ are given by
        \begin{equation}
            B_{\mathsf{v}}=\int_{W_{\mathsf{v}}}p_{\mathsf{e}}\phi_{\mathsf{v}}\phi_{\mathsf{v}}\dd{x}.
        \end{equation}
\end{enumerate}
While matrix $A$ is sparse it is no longer banded as in the traditional FEM for problems on domains, and thus the solution of \eqref{eq:discrete_PDE_matrix} is computationally more intensive.

Similarly to standard error estimates in the FEM framework the $H^1(\mathsf{G})$ error of the finite element solution $u_h$ and the weak solution $u$ is $\mathcal{O}(\hat h)$, where $\hat h:=\max_{\mathsf{e}\in\mathsf{E}}h_{\mathsf{e}}$ and the $L^2(\mathsf{G})$ error is $\mathcal{O}(\hat h^2)$, see \cite[Theorem 3.2]{Arioli2017} for the special case when $c\equiv1$ and \cite[Propositions 6.1-6.2]{Bolin2023} for the general case.

\subsection{Abstract additive Schwarz framework}
In this section we recall the abstract Schwarz framework based on \cite{Dryja1995,Toselli2005}. Let $V$ be a finite dimensional space with the inner product $b(u,v)$ and consider the abstract problem
\begin{equation}\label{eq:asm_ap}
    b(u,v)=f(v),\qquad v\in V.
\end{equation}
Let
\begin{equation}
    V=V_1+V_2+\dots+V_N
\end{equation}
be a not necessarily direct sum of spaces with corresponding symmetric, positive definite bilinear forms $b_i(\cdot,\cdot)$ defined on $V_i\times V_i$. Define the projection-like operators $T_i:V\mapsto V_i$ by
\begin{equation}
    b_i(T_iu,v_i)=b(u,v_i),\qquad v_i\in V_i
\end{equation}
and let
\begin{equation}
    T=T_1+T_2+\dots+T_N.
\end{equation}
Note that if $b_i(u,v)=b(u,v)$ then the operator $T_i$ is equal to the $b(\cdot,\cdot)$-orthogonal projection $P_i$. However, the generality of this framework allows the use of inexact local solvers.

The operator $T$ is used to equivalently reformulate \eqref{eq:asm_ap} as
\begin{equation}\label{eq:asm_Tu}
    Tu=g=\sum_{i=1}^Ng_i=\sum_{i=1}^NT_iu,
\end{equation}
where $g_i$ is obtained by solving
\begin{equation}
    b_i(g_i,v_i)=b(u,v_i)=f(v),\qquad v_i\in V_i.
\end{equation}

The following theorem is the cornerstone of the abstract additive Schwarz framework \cite[Theorem 1]{Dryja1995}.

\begin{thm}\label{theorem:Schwarz}
    Assume that
    \begin{enumerate}[label=(\roman*)]
        \item there exists a constant $C_0>0$ such that there exists a decomposition $u=\sum_{i=1}^Nu_i$ for all $v\in V$, where $u_i\in V_i$, such that
            \begin{equation}
                \sum_{i=1}^Nb_i(u_i,u_i)\le C_0^2b(u,u),
            \end{equation}
        \item there exists a constant $\omega>0$ such that the inequality
            \begin{equation}
                b(u_i,u_i)\le\omega b_i(u_i,u_i),\qquad u_i\in V_i
            \end{equation}
            holds for $i=1,2,\dots,N$,
        \item there exist constants $\epsilon_{ij}\ge0$ such that
            \begin{equation}
                b(u_i,u_j)\le\epsilon_{ij}b^{\frac{1}{2}}(u_i,u_i)b^{\frac{1}{2}}(u_j,u_j),\qquad u_i\in V_i,~u_j\in V_j,
            \end{equation}
            for $i,j=1,2,\dots,N$.
    \end{enumerate}
    Then $T$ is invertible and
    \begin{equation}
        C_0^{-2}b(u,u)\le b(Tu,u)\le\rho(\mathcal{E})\omega b(u,u),\qquad u\in V,
    \end{equation}
    where $\rho(\mathcal{E})$ is the spectral radius of the matrix $\mathcal{E}=\qty{\epsilon_{ij}}_{i,j=1}^N$.
\end{thm}

Theorem \ref{theorem:Schwarz} ensures the existence of a unique solution of \eqref{eq:asm_Tu} and provides the bound $\kappa(T)\le C_0^{-2}\rho(\mathcal{E})\omega$ for the condition number of $T$ w.r.t. the inner product $b(\cdot,\cdot)$, through its Rayleigh quotient. Thus, an upper bound can be computed for the geometric convergence rate of a conjugate gradient or minimal residual method applied to \eqref{eq:asm_Tu}.

\section{Neumann-Neumann method}
We decompose $\mathsf{G}$ to disjoint (w.r.t. its edges) subgraphs $\qty\big{\mathsf{G}_i=(\mathsf{V}_i,\mathsf{E}_i)}_{i=1,2,\dots,N}$. We note that each subgraph is itself a metric graph and that a subgraph may consist of only one edge. The set of vertices that are shared on the boundary of multiple subgraphs will be denoted with $\Gamma$ and called the interface. The corresponding function values are denoted as $u_{\Gamma}=\qty\big(u(\mathsf{v}))_{\mathsf{v}\in\Gamma}$.

\subsection{Continuous version}
The idea of Neumann-Neumann methods is to keep track of the interface values and iteratively update these values based on the deviation from the Neumann-Kirchhoff condition. Formally, we start the algorithm from a zero (or any inexpensive) initial guess $u_{\Gamma}^0$. For $n\ge0$ the new iterate is computed as follows: first we solve the Dirichlet problems

\begin{equation}
    (D_i)\qquad\left\{\begin{aligned}
        f_{\mathsf{e}}(x)&=-(c_{\mathsf{e}}{u_{\mathsf{e}}^{n+\frac{1}{2}}}')'(x)+p_{\mathsf{e}}(x)u_{\mathsf{e}}^{n+\frac{1}{2}}(x),&x\in(0,\ell_{\mathsf{e}}),~\mathsf{e}\in\mathsf{E}_i,~~ & (a)\\
        0&=I_{\mathsf{v}}U_i^{n+\frac{1}{2}}(\mathsf{v}),&\mathsf{v}\in\mathsf{V}_i\backslash\Gamma,~~& (b)\\
        u_{\Gamma}^n(\mathsf{v})&=U_i^{n+\frac{1}{2}}(\mathsf{v}),&\mathsf{v}\in\mathsf{V}_i\cap\Gamma,~~& (c)\\
        0&=C_i(\mathsf{v})^{\top}{U_i^{n+\frac{1}{2}}}'(\mathsf{v}),&\mathsf{v}\in\mathsf{V}_i\backslash\Gamma.~~& (d)
    \end{aligned}
    \right.
\end{equation}
Here the function $C_i$ is the restriction of $C$ to $\mathsf{G}_i$. Note, that we impose natural boundary conditions on the set of vertices $\partial\mathsf{G}_i\cap\partial\mathsf{G}$, but we will still refer to these problems as Dirichlet problems. Then we compute the solutions of the residual Neumann problems
\begin{equation}
    (N_i)\qquad\left\{\begin{aligned}
        0&=-(c_{\mathsf{e}}{w_{\mathsf{e}}^{n+1}}')'(x)+p_{\mathsf{e}}(x)w_{\mathsf{e}}^{n+1}(x),&x\in(0,\ell_{\mathsf{e}}),~\mathsf{e}\in\mathsf{E}_i,~~ & (a)\\
        0&=I_{\mathsf{v}}W_i^{n+\frac{1}{2}}(\mathsf{v}),&\mathsf{v}\in\mathsf{V}_i\backslash\Gamma,~~& (b)\\
        0&=C_i(\mathsf{v})^{\top}{W_i^{n+1}}'(\mathsf{v}),&\mathsf{v}\in\mathsf{V}_i\backslash\Gamma,~~& (c)\\
        &\sum_{i:\mathsf{v}\in\mathsf{V}_i}C_i(\mathsf{v})^{\top}{U_i^{n+\frac{1}{2}}}'(\mathsf{v})=C_i(\mathsf{v})^{\top}{W_i^{n+1}}'(\mathsf{v}),&\mathsf{v}\in\mathsf{V}_i\cap\Gamma.~~&(d)
    \end{aligned}
    \right.
\end{equation}
Finally, we update the interface values as
\begin{equation}\label{eq:uG}
    u_{\Gamma}^{n+1}(\mathsf{v})=u_{\Gamma}^n(\mathsf{v})-\theta\sum_{\mathsf{e}\in\mathsf{E_v}}w_{\mathsf{e}}^{n+1}(\mathsf{v}),\qquad\mathsf{v}\in\Gamma,
\end{equation}
with an appropriate $\theta\in(0,\theta_{\max})$, for some $\theta_{\max}>0$ \cite[Chapter C.3]{Toselli2005}.

\subsection{Discrete version}
In this section we briefly overview some technical tools essential for our subsequent results based on \cite{Mathew2008,Toselli2005}. While in our analysis we will mostly rely on variational notations we will introduce some of the tools in matrix form. For the sake of notational simplicity the following introduction is carried out for a decomposition into two subgraphs.

Let us consider the linear equation $A\boldsymbol{u}=\boldsymbol{f}$ arising from the finite element approximation of an elliptic problem quantum graph $\mathsf{G}=(\mathsf{V},\mathsf{E})$, where $A$ is a symmetric, positive definite matrix. We assume that $\mathsf{G}$ is partitioned into two nonoverlapping subgraphs $\qty\big{\mathsf{G}_i=(\mathsf{V}_i,\mathsf{E}_i)}_{i=1,2}$; that is, we have that
\begin{equation}
    \mathsf{E}=\mathsf{E}_1\cup\mathsf{E}_2,~\mathsf{E}_1\cap\mathsf{E}_2=\emptyset,~\Gamma=\mathsf{V}_1\cap\mathsf{V}_2.
\end{equation}
We recall that in traditional domain decomposition methods we would require that the solution be continuous along the interface and that the normal derivative w.r.t. the domains sum to zero; that is, they are virtually identical to the continuity and Neumann-Kirchhoff conditions at the vertices. We highlight, that while the latter condition is quite natural and has a clear interpretation for quantum graphs, it is not straightforward to define its functional meaning for problems on domains.

\subsubsection{Subassembly and Schur complement systems}
Let us partition the degrees of freedom into those internal to $\mathsf{G}_1$ and to $\mathsf{G}_2$, and those on $\Gamma$ and introduce
\begin{equation}
    A=\begin{bmatrix}A_{II}^{(1)} & 0 & A_{I\Gamma}^{(1)}\\0 & A_{II}^{(2)} & A_{I\Gamma}^{(2)}\\A_{\Gamma I}^{(1)} & A_{\Gamma I}^{(2)} & A_{\Gamma\Gamma}\end{bmatrix},~\boldsymbol{u}=\begin{bmatrix}u_I^{(1)}\\u_I^{(2)}\\u_{\Gamma}\end{bmatrix},~\boldsymbol{f}=\begin{bmatrix}f_I^{(1)}\\f_I^{(2)}\\f_{\Gamma}\end{bmatrix}.
\end{equation}
A crucial observation is that the stiffness matrix $A$ and load vector $f$ can be subassembled from the corresponding components of the (two) subgraphs. If for $i=1,2$ we denote by
\begin{equation}
    f^{(i)}=\begin{bmatrix}f_I^{(i)}\\f_{\Gamma}^{(i)}\end{bmatrix},~A^{(i)}=\begin{bmatrix}A_{II}^{(i)} & A_{I\Gamma}^{(i)}\\A_{\Gamma I}^{(i)} & A_{\Gamma\Gamma}^{(i)}\end{bmatrix}
\end{equation}
the right hand sides and local stiffness matrices of the corresponding elliptic problems with Neumann conditions, then we have that
\begin{equation}
    A_{\Gamma\Gamma}=A_{\Gamma\Gamma}^{(1)}+A_{\Gamma\Gamma}^{(2)},~f_{\Gamma}=f_{\Gamma}^{(1)}+f_{\Gamma}^{(2)}.
\end{equation}
We can find an approximation of the coupled problem as
\begin{equation}\label{eq:coupled}
    \left\{\begin{aligned}
        &A_{II}^{(i)}u_I^{(i)}+A_{I\Gamma}^{(i)}u_{\Gamma}^{(i)}=f_I^{(i)},~&i=1,2\\
        &u_{\Gamma}^{(1)}=u_{\Gamma}^{(2)}=:u_{\Gamma}\\
        &A_{\Gamma I}^{(1)}u_I^{(1)}+A_{\Gamma\Gamma}^{(1)}u_{\Gamma}^{(1)}-f_{\Gamma}^{(1)}=-\qty\big(A_{\Gamma I}^{(2)}u_I^{(2)}+A_{\Gamma\Gamma}^{(2)}u_{\Gamma}^{(2)}-f_{\Gamma}^{(2)})=:\lambda_{\Gamma},
    \end{aligned}
    \right.
\end{equation}
which is equivalent to \eqref{eq:discrete_PDE_matrix}. Clearly if we know the boundary values $u_{\Gamma}$ or the approximate normal derivative $\lambda_{\Gamma}$ the approximate solution inside the domains can be computed by separately solving two Dirichlet or two Neumann problems, respectively. Two well-known corresponding families of domain decomposition algorithms are the Neumann-Neumann and FETI methods. In this article we focus on the former.

To prepare our formal analysis the first standard step of iterative substructuring methods is to eliminate the unknowns $u_I^{(i)}$ with a block factorization
\begin{equation}
    A=\begin{bmatrix}I & 0 & 0\\0 & I & 0\\A_{\Gamma I}^{(1)}{A_{II}^{(1)}}^{-1} & A_{\Gamma I}^{(2)}{A_{II}^{(2)}}^{-1} & I\end{bmatrix}\begin{bmatrix}A_{II}^{(1)} & 0 & A_{I\Gamma}^{(1)}\\0 & A_{II}^{(2)} & A_{I\Gamma}^{(2)}\\ 0 & 0 & S\end{bmatrix},
\end{equation}
where $I$ is the identity matrix and $S=A_{\Gamma\Gamma}-A_{\Gamma I}^{(1)}{A_{II}^{(1)}}^{-1}A_{I\Gamma}^{(1)}-A_{\Gamma I}^{(2)}{A_{II}^{(2)}}^{-1}A_{I\Gamma}^{(2)}$ is the Schur complement relative to the unknowns on $\Gamma$. The corresponding linear system is given by
\begin{equation}
    \begin{bmatrix}A_{II}^{(1)} & 0 & A_{I\Gamma}^{(1)}\\0 & A_{II}^{(2)} & A_{I\Gamma}^{(2)}\\ 0 & 0 & S\end{bmatrix}u=\begin{bmatrix}f_I^{(1)}\\f_I^{(2)}\\g_{\Gamma}\end{bmatrix},
\end{equation}
where $g_{\Gamma}=f_{\Gamma}-A_{\Gamma I}^{(1)}{A_{II}^{(1)}}^{-1}f_I^{(1)}-A_{\Gamma I}^{(2)}{A_{II}^{(2)}}^{-1}f_I^{(2)}$. This can be further reduced to the Schur complement system
\begin{equation}\label{eq:schur}
    Su_{\Gamma}=g_{\Gamma}.
\end{equation}
The fact that $A_{\Gamma\Gamma}$ and $f_{\Gamma}$ can be subassembled from local contributions shows that the same holds for $S$ and $g_{\Gamma}$. Indeed, if for $i=1,2$ we define the local Schur complements by
\begin{equation}
    S^{(i)}:=A_{\Gamma\Gamma}^{(i)}-A_{\Gamma I}^{(i)}{A_{II}^{(i)}}^{-1}A_{I\Gamma}^{(i)}
\end{equation}
and
\begin{equation}
    g_{\Gamma}^{(i)}=f_{\Gamma}^{(i)}-A_{\Gamma I}^{(i)}{A_{II}^{(i)}}^{-1}f_I^{(i)},
\end{equation}
we have that $S=S^{(1)}+S^{(2)}$ and $g_{\Gamma}=g_{\Gamma}^{(1)}+g_{\Gamma}^{(2)}$. We recall the elementary fact that the Schur complement of an invertible block w.r.t. a positive definite matrix is also positive definite. 

Let us define the discrete version of the Neumann-Neumann iteration. Starting from a cheap initial guess $u_{\Gamma}^0$, in an iteration first we solve the Dirichlet problems
\begin{equation}
    (D_i)~~A_{II}^{(i)}u_I^{(i),n+\frac{1}{2}}+A_{I\Gamma}^{(i)}u_{\Gamma}^n=f_I^{(i)},\qquad i=1,2,
\end{equation}
then using the approximation $r_{\Gamma}$ for the flux residual (see the third row of \eqref{eq:coupled}) we solve the Neumann problems
\begin{equation}
    (N_i)~~\begin{bmatrix}A_{II}^{(i)}&A_{I\Gamma}^{(i)}\\A_{\Gamma I}^{(i)}&A_{\Gamma\Gamma}^{(i)}\end{bmatrix}\begin{bmatrix}w_I^{(i),n+1}\\w_{\Gamma}^{(i),n+1}\end{bmatrix}=\begin{bmatrix}0\\r_{\Gamma}\end{bmatrix},\qquad i=1,2.
\end{equation}
Finally, we update the interface values as 
\begin{equation*}
    u_{\Gamma}^{n+1}=u_{\Gamma}^n-\theta\qty\big(w_{\Gamma}^{(1),n+1}+w_{\Gamma}^{(2),n+1}).
\end{equation*}
Eliminating the variables interior to the subdomains of both Dirichlet and Neumann problems shows that
\begin{equation}
    u_{\Gamma}^{n+1}-u_{\Gamma}^n=\theta\qty\Big({S^{(1)}}^{-1}+{S^{(2)}}^{-1})\big(g_{\Gamma}-Su_{\Gamma}^n);
\end{equation}
that is, the Neumann-Neumann algorithm is a preconditioned Richardson iteration for \eqref{eq:schur} using ${S^{(1)}}^{-1}+{S^{(2)}}^{-2}$ as a preconditioner. Often an improved convergence rate can be reached if a further diagonal scaling is used based on the degress of the vertices on $\Gamma$ leading to a preconditioner of the form
\begin{equation*}
    D_{\Gamma}\qty\Big({S^{(1)}}^{-1}+{S^{(2)}}^{-1})D_{\Gamma},
\end{equation*}
where the diagonal elements of $D_{\Gamma}$ are $d_{\mathsf{v}}^{-1}$ for $\mathsf{v}\in\Gamma$. We note that we formulate this Richardson iteration mainly for historical reasons and to avoid the inconvenience of expressing the update of $u_{\Gamma}$ in the case of a more sophisticated iteration. However, in practice, one should instead use a preconditioned conjugate gradient or minimal residual method. Furthermore, the $S^{(i)}$ matrices and especially their inverses should usually not be formed, since we only need to know their effect when applied to a vector. Indeed, instead of multiplying with $S^{(i)}$ (and in particular with the inverse of $A_{II}^{(i)}$) we solve a Dirichlet problem and instead of multiplying with ${S^{(i)}}^{-1}$ we solve a Neumann problem. Other well-known iterative substructuring methods can similarly be characterized by finding a preconditioner for \eqref{eq:schur}. For example, the Dirichlet-Neumann (or Neumann-Dirichlet) corresponds to multiplying the equation with ${S^{(2)}}^{-1}$ (or ${S^{(1)}}^{-1}$).  Then the preconditioned operator ${S^{(2)}}^{-1}S=I+{S^{(2)}}^{-1}S^{(1)}$ corresponds to solving a Dirichlet problem on one subgraph and then solving a Neumann problem on the other.

If we partition $\mathsf{G}$ into many subgraphs a region is called floating if $\partial\mathsf{G}_i\cap\partial\mathsf{G}=\emptyset$. On floating subgraphs Neumann problems of certain elliptic equations, for example if there is no potential, are not uniquely solvable. A possible solution is to use balancing Neumann-Neumann methods, in which we choose a unique solution according to some compatibility condition. In this case the subsequent proof have to be slightly modified, see \cite{Toselli2005} for more details.

Finally, we note that the possibility of domain decomposition was mentioned in \cite{Arioli2017}, where the Schur complement system was solved with conjugate gradient method equipped with diagonal or polynomial preconditioner. These preconditioners are obtained by truncating the Neumann series expansion of
\begin{equation}
    S^{-1}=\qty\big(I-D_S^{-1}(D_S-S))^{-1}D_S^{-1}=\sum_{k=0}^{\infty}\qty\big(D_S^{-1}(D_S-S))^kD_S^{-1}
\end{equation}
to zeroth and first order, respectively, where $D_S$ is a diagonal matrix containing the diagonal elements of $S$. While usually the condition number of the stiffness matrix $A$ is $\mathcal{O}\qty\big(\hat h^{-2})$ and that of the Schur complement $S$ is $\mathcal{O}\qty\big(\hat h^{-1})$, the authors in \cite{Arioli2017} observed that for scale-free graphs the condition number of $S$ seems to be independent of $\hat h$ and proportional to the maximum degree. Furthermore, the dependence on the degree could be rectified with diagonal or polynomial preconditioning. However, these are purely algebraic preconditioners without the formalism of subdomains and without rigorous analysis. Finally, while the assembly of $S$ can be avoided, evaluating $S$ on unit vectors to determine the elements of $D_S$ elicits significant computational costs. 

\subsubsection{Discrete harmonic functions}
The space of discrete harmonic functions is an important subspace of finite element functions and are directly related to the Schur complements and to the interface values $u_{\Gamma}$.

Let us define for $u,v\in V_h(\mathsf{G})$ the bilinear forms corresponding to the global stiffness matrix $A$ and local stiffness matrices $A_i$ as
\begin{equation}
    a(u,v)=\boldsymbol{u}^{\top}A\boldsymbol{v}=\sum_{i=1}^Na^{(i)}(u,v)=\sum_{i=1}^Nu_I^{(i)^{\top}}A^{(i)}v_I^{(i)}.
\end{equation}

A function $u^{(i)}$ defined on $\mathsf{G}_i$ is said to be discrete harmonic on $\mathsf{G}_i$ if
\begin{equation}\label{eq:discrete_harmonic}
    A_{II}^{(i)}u_I^{(i)}+A_{I\Gamma}^{(i)}u_{\Gamma}^{(i)}=0.
\end{equation}
Clearly such a function is completely defined by its values on $\mathsf{V}_i\cap\Gamma$ and it is orthogonal, in the $a_i(\cdot,\cdot)$-inner product, to the space $V_h(\mathsf{G})\cap H_0^1(\mathsf{G}_i,\mathsf{V}_i\cap\Gamma)$, where $H_0^1(\mathsf{G},\mathsf{V}_D)\subset H^1(\mathsf{G})$ is the Sobolev space of functions that vanish on $\mathsf{V}_D\subset\mathsf{V}$. We denote the discrete harmonic extension as $u^{(i)}=:\mathcal{H}_i\qty\big(u_{\Gamma}^{(i)})$.

We denote the space of global, piecewise discrete harmonic functions by $V_h(\Gamma)\subset V_h(\mathsf{G})$, which consists of functions that are discrete harmonic on each subgraph. Based on subassembly arguments a function $u$ is in $V_h(\Gamma)$ if and only if $A_{II}u_I+A_{I\Gamma}u_{\Gamma}=0$ and such a function is completely determined by its values on the interface $\Gamma$. The space $V_h(\Gamma)$ is orthogonal, in the $a(\cdot,\cdot)$-inner product, to each space $V_h\cap H_0^1(\mathsf{G}_i,\mathsf{V}_i\cap\Gamma)$. We denote the piecewise discrete harmonic extension as $u=:\mathcal{H}(u_{\Gamma})$.

In the subsequent analysis the preconditioner will be defined w.r.t. the inner product defined by the Schur complement given by
\begin{equation}
    s(u,v)=u_{\Gamma}^{\top}Sv_{\Gamma}.
\end{equation}
We recall that $s(\cdot,\cdot)$ is symmetric and coercive.

The preceding argument shows that Neumann-Neumann methods can be regarded as computing the global, piecewise discrete harmonic part of the solution of \eqref{eq:discrete_PDE} by defining appropriate preconditioner for the Schur complement $S$. Before we investigate the convergence we must show the equivalence of the interface space, the Schur complement energy and the space of piecewise discrete harmonic functions in $H^1$. The following Lemma shows the energy equivalence of the Schur complement systems and piecewise discrete harmonic functions.
\begin{lemma}\label{lem:Schur_min}
    Let $u_{\Gamma}^{(i)}$ be the restriction of a finite element function to $\mathsf{V}_i\cap\Gamma$. The discrete harmonic extension $u^{(i)}=\mathcal{H}_i\qty\big(u_{\Gamma}^{(i)})$ satisfies
    \begin{equation}
        s_i\qty\big(u^{(i)},u^{(i)})=a_i\qty\big(u^{(i)},u^{(i)})=\min_{v^{(i)}|_{\mathsf{V}_i\cap\Gamma}=u_{\Gamma}^{(i)}}a_i\qty\big(v^{(i)},v^{(i)}).
    \end{equation}
    Similarly, if $u_{\Gamma}$ is the restriction of a finite element function to $\Gamma$, the piecewise discrete harmonic extension $u=\mathcal{H}(u_{\Gamma})$ satisfies
    \begin{equation}\label{eq:schur_min}
        s(u,u)=a(u,u)=\min_{v|_{\Gamma}=u_{\Gamma}}a(v,v).
    \end{equation}
\end{lemma}
\begin{proof}
    The statement follows directly from the definition of (piecewise) discrete harmonic functions in \eqref{eq:discrete_harmonic}.
\end{proof}

We define $d_i=\qty\big|\mathsf{V}_i\cap\Gamma|$ to be the number of vertices of $\mathsf{G}_i$ on the interface and the norm $\norm{\cdot}_{\mathsf{V}_i\cap\Gamma}=\norm{\cdot}_{\mathbb{R}^{d_i}}$.	Let $\mathcal{A}_{i,\max}:H^2(\mathsf{G}_i)\mapsto L^2(\mathsf{G}_i)$ be the operator corresponding to $\mathsf{G}_i$ inherited from $\mathsf{G}$ with $D(\mathcal{A}_{i,\max})=H^2(\mathsf{G}_i)$ and define $\tilde{\mathcal{B}}_i:D(\mathcal{A}_{i,\max})\mapsto\tilde{\mathcal{Y}}_i$ by
\begin{equation}
    \tilde{\mathcal{B}}_iu=\begin{bmatrix}\qty\big(I_{\mathsf{v}}U(\mathsf{v}))_{\mathsf{v}\in\mathsf{V}_i}\\\qty\big(C(\mathsf{v})^{\top}U'(\mathsf{v}))_{\mathsf{v}\in\mathsf{V}_i\backslash\Gamma}\end{bmatrix},\qquad D(\tilde{\mathcal{B}}_i)=D(\mathcal{A}_{i,\max}),
\end{equation}
where $\tilde{\mathcal{Y}}_i=\ell^2(\mathbb{R}^{2m_i-d_i})\eqsim\mathbb{R}^{2m_i-d_i}$. Finally, we define the continuous operator $\tilde{\mathcal{A}}_i:H^2(\mathsf{G}_i)\mapsto L^2(\mathsf{G}_i)$ as
\begin{equation}
    \tilde{\mathcal{A}}_i:=\mathcal{A}_{i,\max},\qquad D(\tilde{\mathcal{A}}_i):=\qty\big{u\in D(\mathcal{A}_{i,\max}):~\tilde{\mathcal{B}}_iu=0_{\tilde{\mathcal{Y}}_i}}.
\end{equation}
That is, a function $u \in D(\tilde{\mathcal{A}}_i)$ is continuous and satisfies the Neumann-Kirchhoff condition at the vertices but not necessarily on the interface $\Gamma$. A function $u\in D(\tilde{\mathcal{A}}_i)$ is said to be harmonic on $\mathsf{G}_i$ if $u\in\mathrm{Ker}(\tilde{\mathcal{A}}_i)$. A function $u\in H^2(\mathsf{G})\cap C(\mathsf{G})$ is said to be piecewise harmonic if $u\big|_{\mathsf{G}_i}\in D(\tilde{\mathcal{A}}_i)\cap\mathrm{Ker}(\tilde{\mathcal{A}}_i)$. Similarly to the discrete case, such a function is expected to be completely determined by the values at $\mathsf{V}_i\cap\Gamma$. The following lemma establishes the existence of the harmonic extension and the equivalence of the interface space and the space of piecewise harmonic functions in $H^2(\mathsf{G}_i)$.

\begin{lemma}\label{lem:equivalent}
    For given boundary data $u_{\Gamma}$ there exists a unique harmonic extension into $\mathsf{G}_i$, and consequently a unique piecewise harmonic extension $u$ into $\mathsf{G}$. Moreover, there exist positive constants $c$ and $C$ such that
    \begin{equation}
        c\norm{u_{\Gamma}}_{\mathsf{V}_i\cap\Gamma}^2\le\norm{u}_{H^2(\mathsf{G}_i)}^2\le C\norm{u_{\Gamma}}_{\mathsf{V}_i\cap\Gamma}^2.
    \end{equation}
\end{lemma}
\begin{proof}
    Let us define the $L:H^2(\mathsf{G}_i)\mapsto\mathbb{R}^{d_i}$ trace operator. Then for any $v\in H^2(\mathsf{G}_i)$ we have that
    \begin{equation}\label{eq:trace_lb}
        \norm{Lv}_{\mathsf{V}_i\cap\Gamma}\le\norm{v}_{L^{\infty}(\mathsf{G}_i)}\le c\norm{v}_{H^1(\mathsf{G}_i)}\le c\norm{v}_{H^2(\mathsf{G}_i)}.
    \end{equation}
    Clearly $\mathcal{A}_0:=\tilde{\mathcal{A}}_i\big|_{\mathrm{Ker}(L)}$ is the generator of a strongly continuous semigroup \cite{Mugnolo2007}, see also \cite[Section 6.5.1]{Mugnolo2014}. We have that $0\in\rho(\mathcal{A}_0)$ since $\mathcal{A}_0$ is invertible, and thus \cite[Lemma 1.2]{Greiner1987} shows that $L\big|_{\mathrm{Ker}(\tilde{\mathcal{A}}_i)}$ is an isomorphism of $\mathrm{Ker}(\tilde{\mathcal{A}}_i)$ onto $\mathbb{R}^{d_i}$; that is, the following inequality holds
    \begin{equation}
        \norm{u}_{H^2(\mathsf{G}_i)}\le C\norm{Lu}_{\mathsf{V}_i\cap\Gamma},
    \end{equation}
    and the proof is finished.
\end{proof}

Finally, the following lemma shows that a similar statement holds for discrete harmonic functions.
\begin{lemma}\label{lem:FEM_equivalent}
    Let $u$ be a piecewise discrete harmonic function on $\mathsf{G}$. Then there exist positive constants $c$ and $C$ independent of $\hat h$ such that
    \begin{equation}
        c\norm{u_{\Gamma}}_{\mathsf{V}_i\cap\Gamma}^2\le\norm{u}_{H^1(\mathsf{G}_i)}^2\le C\norm{u_{\Gamma}}_{\mathsf{V}_i\cap\Gamma}^2.
    \end{equation}
    Consequently, for some positive constants $\tilde c$ and $\tilde C$ independent of $\hat h$, we have that
    \begin{equation}\label{eq:Schur_equivalent}
        \tilde c\sum_{i=1}^N\norm{u_{\Gamma}}_{\mathsf{V}_i\cap\Gamma}^2\le s\qty\big(u,u)\le\tilde C\sum_{i=1}^N\norm{u_{\Gamma}}_{\mathsf{V}_i\cap\Gamma}^2.
    \end{equation}
\end{lemma}
\begin{proof}
    Let $u$ be piecewise discrete harmonic on $\mathsf{G}$ with boundary data $u_{\Gamma}$. The first inequality follows from \eqref{eq:trace_lb}. For the second inequality, let us consider the harmonic extension $v\in H^2(\mathsf{G}_i)$ of $u_{\Gamma}$ into $\mathsf{G}_i$, which uniquely exists in light of Lemma \ref{lem:equivalent}. Furthermore, the function $v$ is continuous and the standard linear interpolation operator $I_h$ can be used resulting in the finite element function $I_hv\in H^1(\mathsf{G}_i)$. Then by \eqref{eq:schur_min} we have that
    \begin{equation}
        \norm{u}_{H^1(\mathsf{G}_i)}\le Ca_i(u,u)\le Ca_i(I_hv,I_hv)\le C\norm{I_hv}_{H^1(\mathsf{G}_i)},
    \end{equation}
    since the $H^1(\mathsf{G}_i)$ norm is equivalent with the $a_i(\cdot,\cdot)$-norm. Furthermore,
    \begin{equation}
        \norm{I_hv}_{H^1(\mathsf{G}_i)}\le\norm{I_hv-v}_{H^1(\mathsf{G}_i)}+\norm{v}_{H^1(\mathsf{G}_i)}\le(C\hat h+1)\norm{v}_{H^2(\mathsf{G}_i)}\le C\norm{u_{\Gamma}}_{\mathsf{V}_i\cap\Gamma}.
    \end{equation}
    The third inequality is shown in the proof of \cite[Theorem 3.2]{Arioli2017} and in the last inequality we used Lemma \ref{lem:equivalent}.
\end{proof}

Let us define $d=|\Gamma|$, the norm $\norm{\cdot}_{\Gamma}=\norm{\cdot}_{\mathbb{R}^d}$ and $d_{\max}=\max_{\mathsf{v}\in\Gamma}\qty\big|\qty\big{j:\mathsf{v}\in\mathsf{V}_j}|$. Then \eqref{eq:Schur_equivalent} implies that
\begin{equation}
    c\norm{u_{\Gamma}}_{\mathbb{R}^d}^2\le s(u,u)\le Cd_{\max}\norm{u_{\Gamma}}_{\mathbb{R}^d}^2.
\end{equation}
The following statement is an immediate consequence.

\begin{cor}\label{corollary:S}
    The condition number of the Schur complement $S$ is a constant that is independent of $\hat h$ and satisfies the bound $\kappa(S)\le Cd_{\max}$.
\end{cor}

We note that this phenomenon was already observed, although not rigorously investigated, for scale-free graphs in \cite{Arioli2017}.

\subsection{Schwarz iteration}
With the above auxiliary results we can reformulate the Neumann-Neumann method as an abstract additive Schwarz iteration. We choose $V=V_h(\Gamma)$ and $V_i=V_i(\Gamma)$, where $V_i(\Gamma)\subset V_h(\Gamma)$ denotes the subspace of discrete harmonic functions that vanish on $\Gamma\backslash\mathsf{V}_i$. For the bilinear forms we set $b(u,v)=s(u,v)$ on $V\times V$ and
\begin{equation}
    b_i(u,v)=s_i\qty\big(I_h(\nu_iu),I_h(\nu_iv))=a_i\qty\big(\mathcal{H}_i(\nu_iu),\mathcal{H}_i(\nu_iv))
\end{equation}
on $V_i\times V_i$. The counting functions $\nu_i$ are defined on $\Gamma\cup\partial\mathsf{G}$ by
\begin{equation}
    \nu_i(\mathsf{v})=\begin{cases}\qty\big|\qty\big{j:\mathsf{v}\in\mathsf{V}_j}|,\quad&\mathsf{v}\in(\Gamma\cap\mathsf{V}_i)\cup\partial\mathsf{G}_i,\\0,&\mathsf{v}\in\Gamma\backslash\mathsf{V}_i.\end{cases}
\end{equation}
The pseudoinverses $\nu_i^{\dag}$ of the $\nu_i$ functions, given as
\begin{equation}
    \nu_i^{\dag}(v)=\begin{cases}\nu_i^{-1}(\mathsf{v}),\quad&\mathsf{v}\in(\Gamma\cap\mathsf{V}_i)\cup\partial\mathsf{G}_i,\\0,&\mathsf{v}\in\Gamma\backslash\mathsf{V}_i,\end{cases}
\end{equation}
define a partition of unity on $\Gamma\cup\partial\mathsf{G}$; that is,
\begin{equation}
    \sum_{i=1}^N\nu_i^{\dag}(\mathsf{v})\equiv1,\qquad\mathsf{v}\in\Gamma\cup\partial\mathsf{G}.
\end{equation}
Finally, the operators $T_i:V\mapsto V_i$ are defined by
\begin{equation}
    b_i(T_iu,v)=b(u,v),\qquad v\in V_i,
\end{equation}
and the operator $T$ by
\begin{equation}\label{eq:T}
    T=T_1+T_2+\dots+T_N.
\end{equation}

\begin{prop}\label{proposition:T}
    The operator $T$ defined by \eqref{eq:T} is invertible and for all $u\in V$ the following inequality holds
    \begin{equation}
        \gamma_0s(u,u)\le s(Tu,u)\le\gamma_1\rho(\mathcal{E})s(u,u),
    \end{equation}
    where $\gamma_0$ and $\gamma_1$ are constants independent of $\hat h$, where $\mathcal{E}=\qty{\epsilon_{ij}}_{i,j=1}^N$ is defined elementwise by
    \begin{equation}
        \epsilon_{ij}=\begin{cases}1,\quad&\mathsf{V}_i\cap\mathsf{V}_j\neq\emptyset,\\0,\quad&\text{otherwise.}\end{cases}
    \end{equation}
\end{prop}
\begin{proof}
    We have to establish the three estimates of Theorem \ref{theorem:Schwarz}.

    Assumption (i): For $u\in V$ we choose $u_i=I_h\qty\big(\nu_i^{\dag}u)$, $i=1,2,\dots,N$. Clearly $u_i\in V_i$ and $u=\sum_{i=1}^Nu_i$ holds, and
    \begin{equation}
        b_i(u_i,u_i)=a_i(\mathcal{H}_iu,\mathcal{H}_iu)=a_i(u,u).
    \end{equation}
    By subassembly, this shows that
    \begin{equation}
        \sum_{i=1}^Nb_i(u_i,u_i)=a(u,u)=s(u,u)=b(u,u).
    \end{equation}

    Assumption (ii): For $u_i\in V_i$ we have that
    \begin{equation}
        s(u_i,u_i)=s_i(u_i,u_i)+\sum_{j:\mathsf{V}_j\cap\mathsf{V}_i\neq\emptyset}s_j(u_i,u_i).
    \end{equation}
    Using Lemma \ref{lem:FEM_equivalent} shows that $s_i(u_i,u_i)\le C\norm{u_i}_{\mathsf{V}_i\cap\Gamma}$ and that
    \begin{equation}
        s_j(u_i,u_i)\le C\norm{u_i}_{\mathsf{V}_j\cap\Gamma}^2\le C\norm{u_i}_{\mathsf{V}_i\cap\Gamma}^2,
    \end{equation}
    since $u_i\in V_i$, and thus $u_i(x)=0$ for $x\in(\mathsf{V}_j\cap\Gamma)\backslash\mathsf{V}_i$. Using Sobolev's embedding we can further bound $\norm{u_i}_{\mathsf{V}_i\cap\Gamma}^2$ as
    \begin{equation}
        \begin{aligned}
            \norm{u_i}_{\mathsf{V}_i\cap\Gamma}^2&\le C\norm{u_i}_{L^{\infty}(\mathsf{G}_i)}^2\le C\norm{u_i}_{H^1(\mathsf{G}_i)}^2\le Ca_i(u_i,u_i)\\
            &=Cs_i(u_i,u_i)\le Cs_i\qty\big(I_h(\nu_iu_i),I_h(\nu_iu_i))=Cb_i(u_i,u_i).
        \end{aligned}
    \end{equation}
    Combining the above yields $b(u_i,u_i)\le Cb_i(u_i,u_i)$ for $u_i\in V_i$ as required.

    Assumption (iii): It is easy to see that
    \begin{equation}
        \epsilon_{ij}=\begin{cases}1,\quad&\mathsf{V}_i\cap\mathsf{V}_j\neq\emptyset,\\0,\quad&\text{otherwise,}\end{cases}
    \end{equation}
    as $V_i\cap V_j\neq\emptyset$ if and only if $\mathsf{V}_i\cap\mathsf{V}_j\neq\emptyset$.
\end{proof}

This shows that the condition number of the preconditioned system is independent of $\hat h$. We note that $\rho(\mathcal{E})\le d_{\max}$ via the Gershrogin theorem. Finally, we state our main theorem.

\begin{thm}
    The Neumann-Neumann algorithm converges to the solution of \eqref{eq:discrete_PDE_matrix} with a geometric rate that is independent of $\hat h$.
\end{thm}
\begin{proof}
    The statement follows from Proposition \ref{proposition:T} and Lemma \ref{lem:FEM_equivalent}.
\end{proof}

\begin{rem}
    We note that in a multidimensional setting one usually assumes that the substructures and the elements are shape regular, meaning that the number of neighbours of any subdomain, and thus $\rho(\mathcal{E})$, is bounded by a constant. Furthermore, the verification of assumption (i) and (ii) is more challenging, and accordingly the estimates on $\frac{s(Tu,u)}{s(u,u)}$ are more complicated. In particular, usually polylogarithmic bounds of the form $\tilde h^{-2}\qty\Big(1+\log\frac{\tilde h}{\hat h})^2$ appear, where $\tilde h$ denotes the size of a typical subdomain, see \cite{Dryja1995,Toselli2005}. The main technical difficulty is the fact that the boundary spaces of the domains are equipped with the $H^{\frac{1}{2}}$ Sobolev-Slobodeckij seminorm, which cannot be so straightforwardly estimated as in our case.
\end{rem}

\section{Numerical experiments}\label{sec:num}
In this section we introduce and discuss some numerical experiments. The C$++$ implementation mainly relies on Eigen 3.4.0 and is compiled with GCC 13.2.1. The graphs are generated with NetworkX 3.1 in Python 3.11.6. The experiments have been performed on a computer with Intel(R) Core(TM) i7-8565U CPU @ 1.80GHz and 16 GB of RAM in Python 3.11.6.  While our convergence theory holds for arbitrary (nonoverlapping) decomposition, in all experiments, we decompose the quantum graph to its edges. The Schur complement problems are solved with BiCGSTAB without preconditioning, with diagonal preconditioning, with first-degree polynomial preconditioning and finally with Neumann-Neumann preconditioning. While Corollary \ref{corollary:S} shows that condition number of the Schur complement is independent of $\hat h$, it might still increase as the number of vertices grows, as indicated by the results below. Interestingly, this dependence is already somewhat mitigated with a diagonal preconditioner and seemingly eliminated with a polynomial or Neumann-Neumann preconditioner. We found that for small graphs with $|\mathsf{V}|\ll1000$ solving the Schur complement system without preconditioning is the fastest independently of $\hat h$, but for larger graphs preconditioning is more and more crucial as $\log_2\qty\big(\hat h^{-1})$ increases. We note that while a single (BiCGSTAB) iteration with the diagonal or the polynomial preconditioner is cheaper than an iteration with the Neumann-Neumann preconditioner the former approaches are still slower in these cases since the diagonal of the Schur complement has to be computed first. While the performance of these methods may depend on various implementation factors, the following experiments clearly show that the runtime of the diagonal and the polynomial preconditioners blow up as the size of the graph of $\log_2\qty\big(\hat h^{-1})$ increase.

The initial guess is set to the zero vector and the iteration is stopped after the relative residual norm reduces below the square root of the machine precision $\varepsilon\approx2.2204\cdot10^{-16}$.

\subsection{Dorogovtsev-Goltsev-Mendes graphs}
The first set of test graphs are a family of scale-free planar graphs introduced in \cite{Dorogovtsev2002}, defined iteratively as follows. The graph $\mathsf{DGM}(0)$ is the path graph with two vertices. The graph $\mathsf{DGM}(n+1)$ is generated from $\mathsf{DGM}(n)$ by adding a new vertex for each edge and connecting it with the endpoint of the edge. The graph $\mathsf{DGM}(n)$ has $|\mathsf{V}|=\frac{3}{2}\qty\big(3^n+1)$ and $|\mathsf{E}|=3^n$. Figure \ref{fig:russki} shows the first few graphs of this iteration. First we set $\log_2\qty\big(\hat h^{-1})=6$ and apply BiCGSTAB to the Schur complement system of $\mathsf{DGM}$ graphs of increasing size. Tables \ref{tab:russki:size:iter} and \ref{tab:russki:size:runtime} show the number of iterations and the runtime, respectively, without preconditioning and with diagonal, polynomial and Neumann-Neumann preconditioning. Tables \ref{tab:russki:h:iter} and \ref{tab:russki:h:runtime} show the same for $\mathsf{DGM}(7)$ with increasing $\log_2\qty\big(\hat h^{-1})$.

\begin{figure}[H]
    \begin{center}
        \begin{tikzpicture}[scale=0.5, every node/.style={scale=1}]
            \begin{scope}
                \node[draw,shape=circle,fill=black,inner sep=1pt] (1) at (90:1) {};
                \node[draw,shape=circle,fill=black,inner sep=1pt] (2) at (210:1) {};
                \node[draw,shape=circle,fill=black,inner sep=1pt] (3) at (330:1) {};
                \path[-] (1) edge (2);
                \path[-] (2) edge (3);
                \path[-] (3) edge (1);
            \end{scope}
            \begin{scope}[xshift=4cm,yshift=0.35566cm]
                \node[draw,shape=circle,fill=black,inner sep=1pt] (1) at (30:1) {};
                \node[draw,shape=circle,fill=black,inner sep=1pt] (2) at (90:1) {};
                \node[draw,shape=circle,fill=black,inner sep=1pt] (3) at (150:1) {};
                \node[draw,shape=circle,fill=black,inner sep=1pt] (4) at (210:1) {};
                \node[draw,shape=circle,fill=black,inner sep=1pt] (5) at (270:1) {};
                \node[draw,shape=circle,fill=black,inner sep=1pt] (6) at (330:1) {};
                \path[-] (1) edge (2);
                \path[-] (2) edge (3);
                \path[-] (3) edge (4);
                \path[-] (4) edge (5);
                \path[-] (5) edge (6);
                \path[-] (6) edge (1);
                \path[-] (2) edge (4);
                \path[-] (4) edge (6);
                \path[-] (6) edge (2);
            \end{scope}
            \begin{scope}[xshift=8.8cm,yshift=0.64438cm]
                \node[draw,shape=circle,fill=black,inner sep=1pt] (11) at (30:1) {};
                \node[draw,shape=circle,fill=black,inner sep=1pt] (12) at (90:1) {};
                \node[draw,shape=circle,fill=black,inner sep=1pt] (13) at (150:1) {};
                \node[draw,shape=circle,fill=black,inner sep=1pt] (14) at (210:1) {};
                \node[draw,shape=circle,fill=black,inner sep=1pt] (15) at (270:1) {};
                \node[draw,shape=circle,fill=black,inner sep=1pt] (16) at (330:1) {};
                \node[draw,shape=circle,fill=black,inner sep=1pt] (21) at ([shift=(30:1)]30:1) {};
                \node[draw,shape=circle,fill=black,inner sep=1pt] (22) at ([shift=(30:1)]90:1) {};
                \node[draw,shape=circle,fill=black,inner sep=1pt] (26) at ([shift=(30:1)]330:1) {};
                \node[draw,shape=circle,fill=black,inner sep=1pt] (32) at ([shift=(150:1)]90:1) {};
                \node[draw,shape=circle,fill=black,inner sep=1pt] (33) at ([shift=(150:1)]150:1) {};
                \node[draw,shape=circle,fill=black,inner sep=1pt] (34) at ([shift=(150:1)]210:1) {};
                \node[draw,shape=circle,fill=black,inner sep=1pt] (44) at ([shift=(270:1)]210:1) {};
                \node[draw,shape=circle,fill=black,inner sep=1pt] (45) at ([shift=(270:1)]270:1) {};
                \node[draw,shape=circle,fill=black,inner sep=1pt] (46) at ([shift=(270:1)]330:1) {};
                \path[-] (11) edge (12);
                \path[-] (12) edge (13);
                \path[-] (13) edge (14);
                \path[-] (14) edge (15);
                \path[-] (15) edge (16);
                \path[-] (16) edge (11);
                \path[-] (12) edge (14);
                \path[-] (14) edge (16);
                \path[-] (16) edge (12);
                \path[-] (12) edge (21);
                \path[-] (12) edge (22);
                \path[-] (16) edge (21);
                \path[-] (16) edge (26);
                \path[-] (21) edge (22);
                \path[-] (26) edge (21);
                \path[-] (12) edge (32);
                \path[-] (12) edge (33);
                \path[-] (14) edge (33);
                \path[-] (14) edge (34);
                \path[-] (32) edge (33);
                \path[-] (33) edge (34);
                \path[-] (14) edge (44);
                \path[-] (14) edge (45);
                \path[-] (16) edge (45);
                \path[-] (16) edge (46);
                \path[-] (44) edge (45);
                \path[-] (45) edge (46);
            \end{scope}
        \end{tikzpicture}
    \end{center}
    \caption{The graphs $\mathsf{DGM}(1)$, $\mathsf{DGM}(2)$ and $\mathsf{DGM}(3)$.}
    \label{fig:russki}
\end{figure}

\begin{center}
    \begin{threeparttable}
        \begin{tabular}{ m{4.5em} c c c c }
            \hline
            Graph & No prec. & Diagonal & Polynomial & Neumann-Neumann\\
            \hline
            $\mathsf{DGM}(5)$ & 25 & 8 & 7 & 7\\
            $\mathsf{DGM}(6)$ & 42 & 11 & 9 & 9\\
            $\mathsf{DGM}(7)$ & 86 & 15 & 11 & 11\\
            $\mathsf{DGM}(8)$ & 145 & 18 & 13 & 12\\
            $\mathsf{DGM}(9)$ & 244 & 21 & 14 & 14\\
            \hline
        \end{tabular}
        \caption{Number of BiCGSTAB iterations for the Schur complement systems of Dorogovtsev-Goltsev-Mendes graphs of increasing size with $\log_2\qty\big(\hat h^{-1})=6$.}
        \label{tab:russki:size:iter}
    \end{threeparttable}
\end{center}

\begin{center}
    \begin{threeparttable}
        \begin{tabular}{ m{4.5em} c c c c }
            \hline
            Graph & No prec. & Diagonal & Polynomial & Neumann-Neumann\\
            \hline
            $\mathsf{DGM}(5)$ & $0.0277~s$ & $0.0448~s$ & $0.0387~s$ & $0.0672~s$\\
            $\mathsf{DGM}(6)$ & $0.0931~s$ & $0.2528~s$ & $0.2675~s$ & $0.2468~s$\\
            $\mathsf{DGM}(7)$ & $0.4907~s$ & $2.0924~s$ & $2.1583~s$ & $1.0146~s$\\
            $\mathsf{DGM}(8)$ & $2.4740~s$ & $19.3271~s$ & $19.4995~s$ & $5.2082~s$\\
            $\mathsf{DGM}(9)$ & $12.6867~s$ & $184.5470~s$ & $185.9770~s$ & $38.8155~s$\\
            \hline
        \end{tabular}
        \caption{Runtime of BiCGSTAB iteration for the Schur complement systems of Dorogovtsev-Goltsev-Mendes graphs of increasing size with $\log_2\qty\big(\hat h^{-1})=6$.}
        \label{tab:russki:size:runtime}
    \end{threeparttable}
\end{center}

\begin{center}
    \begin{threeparttable}
        \begin{tabular}{ m{4.5em} c c c c }
            \hline
            $\log_2\qty\big(\hat h^{-1})$ & No prec. & Diagonal & Polynomial & Neumann-Neumann\\
            \hline
            4 & 81 & 14 & 11 & 11\\
            6 & 86 & 15 & 11 & 11\\
            8 & 72 & 15 & 11 & 11\\
            10 & 86 & 14 & 11 & 11\\
            12 & 87 & 15 & 11 & 11\\
            \hline
        \end{tabular}
        \caption{Number of BiCGSTAB iterations for the Schur complement system of $\mathsf{DGM}(7)$ with increasingly finer meshes.} 
        \label{tab:russki:h:iter}
    \end{threeparttable}
\end{center}

\begin{center}
    \begin{threeparttable}
        \begin{tabular}{ m{4.5em} c c c c }
            \hline
            $\log_2\qty\big(\hat h^{-1})$ & No prec. & Diagonal & Polynomial & Neumann-Neumann\\
            \hline
            4 & $0.1038~s$ & $0.4758~s$ & $0.4829~s$ & $0.6264~s$\\
            6 & $0.4853~s$ & $2.0668~s$ & $2.0902~s$ & $0.9900~s$\\
            8 & $1.8201~s$ & $9.3257~s$ & $9.4613~s$ & $2.1968~s$\\
            10 & $8.8582~s$ & $37.9166~s$ & $38.7955~s$ & $6.7927~s$\\
            12 & $42.9914~s$ & $198.7080~s$ & $202.073~s$ & $26.2862~s$\\
            \hline
        \end{tabular}
        \caption{Runtime of BiCGSTAB iteration for the Schur complement system of $\mathsf{DGM}(7)$ with increasingly finer meshes.} 
        \label{tab:russki:h:runtime}
    \end{threeparttable}
\end{center}

\subsection{Barab\'asi-Albert model}
Next, we test our method on scale-free graphs with $|\mathsf{E}|\approx2|\mathsf{V}|$ generated using the Barab\'asi-Albert model \cite{Barabasi1999}. Unlike the $\mathsf{DGM}$ graphs, which are generated deterministically, the Barab\'asi-Albert model has randomness involved, and thus the following results have to be understood in a probabilistic sense.

Again, we set $\log_2\qty\big(\hat h^{-1})=6$ and apply BiCGSTAB to the Schur complement system of scale-free graphs of increasing size. Tables \ref{tab:scale-free:size:iter} and \ref{tab:scale-free:size:runtime} show the number of iterations and the runtime, respectively, without preconditioning and with diagonal, polynomial and Neumann-Neumann preconditioning. Tables \ref{tab:scale-free:h:iter} and \ref{tab:scale-free:h:runtime} show the same for $\mathsf{SF}(1000)$ with increasing $\log_2\qty\big(\hat h^{-1})$.

\begin{center}
    \begin{threeparttable}
        \begin{tabular}{ m{4.5em} c c c c }
            \hline
            Graph & No prec. & Diagonal & Polynomial & Neumann-Neumann\\
            \hline
            $\mathsf{SF}(100)$ & 28 & 18 & 9 & 9\\
            $\mathsf{SF}(500)$ & 47 & 21 & 10 & 10\\
            $\mathsf{SF}(1000)$ & 57 & 18 & 10 & 10\\
            $\mathsf{SF}(2000)$ & 68 & 20 & 10 & 10\\
            $\mathsf{SF}(5000)$ & 83 & 20 & 10 & 10\\
            \hline
        \end{tabular}
        \caption{Number of BiCGSTAB iterations for the Schur complement systems of scale-free graphs of increasing size with $\log_2\qty\big(\hat h^{-1})=6$.}
        \label{tab:scale-free:size:iter}
    \end{threeparttable}
\end{center}

\begin{center}
    \begin{threeparttable}
        \begin{tabular}{ m{4.5em} c c c c }
            \hline
            Graph & No prec. & Diagonal & Polynomial & Neumann-Neumann\\
            \hline
            $\mathsf{SF}(100)$ & $0.0198~s$ & $0.0285~s$ & $0.0320~s$ & $0.1244~s$\\
            $\mathsf{SF}(500)$ & $0.1373~s$ & $0.4779~s$ & $0.4694~s$ & $0.3496~s$\\
            $\mathsf{SF}(1000)$ & $0.3049~s$ & $1.7175~s$ & $1.7336~s$ & $0.8347~s$\\
            $\mathsf{SF}(2000)$ & $0.7321~s$ & $6.9985~s$ & $7.1694~s$ & $2.0819~s$\\
            $\mathsf{SF}(5000)$ & $2.2780~s$ & $46.1882~s$ & $46.2169~s$ & $9.1687~s$\\
            \hline
        \end{tabular}
        \caption{Runtime of BiCGSTAB iteration for the Schur complement systems of scale-free graphs of increasing size with $\log_2\qty\big(\hat h^{-1})=6$.}
        \label{tab:scale-free:size:runtime}
    \end{threeparttable}
\end{center}

\begin{center}
    \begin{threeparttable}
        \begin{tabular}{ m{4.5em} c c c c }
            \hline
            $\log_2\qty\big(\hat h^{-1})$ & No prec. & Diagonal & Polynomial & Neumann-Neumann\\
            \hline
            4 & 47 & 19 & 9 & 9\\
            6 & 44 & 19 & 9 & 9\\
            8 & 46 & 17 & 9 & 9\\
            10 & 46 & 19 & 9 & 9\\
            12 & 49 & 20 & 9 & 9\\
            \hline
        \end{tabular}
        \caption{Number of BiCGSTAB iterations for the Schur complement system of $\mathsf{SF}(100)$ with increasingly finer meshes.} 
        \label{tab:scale-free:h:iter}
    \end{threeparttable}
\end{center}

\begin{center}
    \begin{threeparttable}
        \begin{tabular}{ m{4.5em} c c c c }
            \hline
            $\log_2\qty\big(\hat h^{-1})$ & No prec. & Diagonal & Polynomial & Neumann-Neumann\\
            \hline
            4 & $0.0642~s$ & $0.3847~s$ & $0.4120~s$ & $0.4987~s$\\
            6 & $0.2891~s$ & $1.7520~s$ & $1.7520~s$ & $0.8524~s$\\
            8 & $1.3794~s$ & $7.8648~s$ & $7.8611~s$ & $1.9834~s$\\
            10 & $5.3083~s$ & $32.3043~s$ & $32.4282~s$ & $5.9379~s$\\
            12 & $25.7713~s$ & $167.5980~s$ & $167.9090~s$ & $22.7556~s$\\
            \hline
        \end{tabular}
        \caption{Runtime of BiCGSTAB iteration for the Schur complement system of $\mathsf{SF}(1000)$ with increasingly finer meshes.} 
        \label{tab:scale-free:h:runtime}
    \end{threeparttable}
\end{center}

%\bibliographystyle{abbrv}
%\bibliography{References}

\end{document}